\documentclass [12pt]{article}
\usepackage{amsmath}
\usepackage{amsthm}
\usepackage{graphicx}
\usepackage{amsfonts, amssymb}
\usepackage[utf8]{inputenc} 
\usepackage[english]{babel}

\usepackage[justification=centering]{caption}
\usepackage{pgf,tikz}
\usetikzlibrary{arrows}

\newtheorem{thm}{Theorem}
\newtheorem{lemma}[thm]{Lemma}
\newtheorem{prop}[thm]{Proposition}
\newtheorem{cor}[thm]{Corollary}
\newtheorem{remark}[thm]{Remark}

\def\R{{\mathbb R}}
\def\E{{\mathbb E\,}}

\def\P{{\mathbb P}}
\def\z{{\mathbb Z}}

\def\n{{\mathbb N}}

\def\N{{\mathbb N}}

\def\AA{{\mathcal A}}
\def\BB{{\mathcal B}}
\def\FF{{\mathcal F}}
\def\GG{{\mathcal G}}

\def\w{{\mathcal W}}
\def\M{{\mathcal M}}
\def\ww{{\widetilde{\w}}}
\def\pp{{\widetilde{p}}}

\def\dd{{\widetilde{\Delta}}}
\def\aa{{\widetilde{A}}}
\def\bb{{\widetilde{B}}}
\def\www{{\widetilde{w}}}
\def\q{{\widetilde{q}}}
\def\L{{\widetilde{L}}}

\begin{document}
\title{Existence of a persistent hub in the convex preferential attachment model}
\author
{Pavel Galashin}
\maketitle
\abstract{ 
A vertex of a randomly growing graph is called a persistent hub if at all but finitely many moments of time it has the maximal degree in the graph. We establish the existence of a persistent hub in the Barabási--Albert random graph model with probability one. We also extend this result to the class of convex preferential attachment graphs, where a vertex of degree $k$ gets a new edge with probability proportional to some convex function of $k$.
}

\section{Introduction}
\label{sect:intro}

The preferential attachment model was introduced by R. Albert and A. L. Barabási in \cite{BA} in order to create a natural model for a dynamically growing random network with a scale-free power-law distribution of degrees of vertices. This distribution appears in many large real random graphs such as internet, social networks, etc.

Since then the model became very popular and has been investigated mathematically and empirically in many works, for example \cite{BRST,KR,M1,N,GGL}. Many generalizations have been suggested: \cite{ACL,DM,M2} etc.

The Barabási--Albert (BA) preferential attachment model is defined as follows:
\begin{enumerate}
  \item Before the first step we have a natural number $m_0$ and a tree, which contains one vertex $v_1$ and zero edges.
  \item On the $k$-th step ($k \ge 1$) we attach one new vertex $v_k$ and $m_0$ new edges to the graph. These edges are attached one-by-one. Every edge connects $v_k$ to some old vertex $v_i$ of the graph. This vertex $v_i$ is chosen randomly, with probability proportional to its degree $\deg(v_i)$. The degrees are being refreshed after every edge attachment.
\end{enumerate}

The convex preferential attachment model involves the same algorithm, but the old vertex $v_i$ is chosen with probability proportional to $\w(\deg(v_i))$, where $\w:\z_+\to\R_+$ is a fixed positive, convex and unbounded function defined on the set of non-negative integers. Formal definitions of these models are given below.

\subsection{The main result}
Investigation of vertices of maximal degree became one of the most popular research directions in preferential attachment, because the presence of vertices with large degrees is one of the features of preferential attachment model as opposed to the classical Erdős–Rényi model. While the graph grows, different vertices can have maximal degrees at different steps.

If a vertex has the maximal degree on some step, then it is called \textit{a hub}. If some vertex is a hub for all but finitely many steps, then it is called \textit{a persistent hub}. This notion has appeared in the literature, see, for example, \cite{DM}.

Now we can ask the following question: does the hub change infinitely many times, or does there appear a persistent hub after some number of steps?

The following theorem answers this question:

\begin{thm}\label{thm:main}
In the BA and convex preferential attachment models with probability $1$ there exist numbers $n$ and $k$ such that on any step after the $n$-th step the vertex $v_k$ has the highest degree among all vertices. In other words, the persistent hub appears with probability one.
\end{thm}

\subsection{Previous research}
S.~Dereich and P.~M{\"o}rters consider in \cite{DM} a model very similar to the BA preferential attachment model with $m_0=1$: one starts with a single vertex and a fixed concave function $f:\z_+\to\R_+$, such that $f(n)\leq n+1$ for all $n\in\z_+$. On the $k$-th step:
\begin{itemize}
  \item one new vertex $v_k$ is added
  \item for every old vertex $v_i$ the edge $e_{ik}$ connecting $v_i$ to $v_k$ is added with probability
      $$\frac{f(\deg(v_i))}{k}.$$
\end{itemize}
The main difference is that here for $i\neq j$ the decisions of adding the edges $e_{ik}$ and $e_{jk}$ are made \textbf{independently}, while in the BA model exactly one (if $m_0=1$) old vertex becomes connected to the new vertex.

The question of existence of a persistent hub for this model has been answered completely in \cite{DM}, but the difference between the models makes it too hard to apply those results to the BA case.

Theorem 1.7  from \cite{DM} states that a persistent hub appears if and only if
$$\sum_{k=1}^\infty\frac1{f(k)^2}<\infty.$$
This condition is of course satisfied if $f$ is convex and unbounded. But, even though the models are different, it remains an interesting open question if the same result also holds in our situation after we drop the convexity assumption. A weaker conjecture is that a persistent hub in the BA model appears if the weight function is superlinear.

\subsection{Outline of the paper}

The proof of Theorem \ref{thm:main} assumes $m_0=1$. We extend it to the case of an arbitrary $m_0$ in the last subsection.

\begin{enumerate}
\item In Subsection \ref{subsect:specs} we give precise specifications of all considered models.

  \item In Section \ref{sect:2dim} we investigate the joint behavior of two fixed vertices. The main result of this section is Proposition \ref{thm:qAbound}, which states that if on the $n$-th step there is some vertex $v_k$ with a high degree, then with high probability the degree of $v_{n+1}$ will be lower than the degree of $v_k$ on every step.

  \item In Section \ref{sect:main} we prove the main result:
  \begin{enumerate}
  \item In Subsection \ref{subsect:lemmas} we prepare some tools to make a comparison between the convex model and the BA model.
    \item
    In Subsection \ref{subsect:tooFast}  we show that the maximal degree tends to infinity fast enough, even though the degree of any fixed vertex may be bounded.
      \item In Subsection \ref{subsect:finiteCandidates} we use Borel--Cantelli lemma to show that with probability $1$ all but finite number of vertices never have the highest degree.

  \item For every pair of the remaining vertices we prove in Subsection \ref{subsect:finitePair} that there is only a finite number of steps on which their degrees coincide. This completes the proof for the case $m_0=1$.
  \item In Subsection \ref{subsect:m0} we modify the steps of the proof so that they work for an arbitrary $m_0$.
\end{enumerate}
\end{enumerate}
Throughout the paper every weight function is assumed to be convex.
\subsection{Definitions of the models}
\label{subsect:specs}
First we define all models for $m_0=1$, because the major part of the proof is concentrated around this case.
\subsubsection{Basic model for $m_0=1$}
Let $X_n^k$ be the degree of the vertex $v_k$ before the $n$-th step.
Note that before the $n$-th step there are $n$ vertices and $n-1$ edges in the tree, and the total degree of all vertices has a very simple form:

\[
   \sum_{k=1}^{n} X_n^k= 2(n-1).
\]

By $p_n^k$ denote the probability that the new edge on the $n$-th step is attached to the vertex $v_k,\ k \le n$. Then
\[
   p_n^k := \begin{cases}
   1,& k=n=1, \\
   \frac{X_n^k}{2(n-1)}, & n>1, 1\le k \le n.
   \end{cases}
\]

Since in this paper we are not interested in the topological structure of the tree, we can just consider the Markov chain of vectors
$X_n:=(X_n^k)_{1\le k\le n}$.

\subsubsection{Generalized model for $m_0=1$}

\label{definition}
Let $\w:\z_+ \to \R_+$ be a strictly positive function.

In this model, a vertex $v_k$ of degree $X_n^k$ has weight $\w(X_n^k)$, and the probability $p_n^k$ that the new edge is attached to the vertex $v_k$ on the $n$-th step  is defined as a ratio of the weight of $v_k$ to the total weight of all vertices:
\[
   p_n^k := \begin{cases}
   1,& k=n=1, \\
   \frac{\w(X_n^k)}{w_n}, & n>1, 1\le k \le n.
   \end{cases}
\]
where
\[
	w_n:=\sum_{k=1}^{n} \w(X_n^k).
\]
(here, unlike the basic model, $w_n$ is a random variable)

This model is also common, for example, in \cite{DM}, \cite{OS} the cases of superlinear ($\w(n)\gg n$) and sublinear ($\w(n)\ll n$) preferential attachment are considered, in \cite{RTV} the asymptotical degree distribution for a wide range of weight functions is given, and in \cite{RT} the Hausdorff dimension of some natural measure on the leaves of the limiting tree is evaluated.

\subsubsection{Convex model for $m_0=1$}
The convex model is a special case of the generalized model. Here $\w(n)$ must be convex and unbounded. Note that $\w(n)$ is not assumed to be increasing.

The convex model includes several popular special cases. We have already discussed the basic model. In \cite{OS} the case $\w(n)=n^p,\ p>1$, is considered. In \cite{M2} and \cite{M1} the case $\w(n)=n+\beta,\ \beta>-1$ is considered. We call this case {\it the linear model}. Some similar models have been considered earlier, for example, in \cite{DMconcave} the concave preferential attachment rule is investigated.

The convexity condition is pretty mild, but, on the other hand, it is very convenient and simplifies proofs and calculations.

\subsubsection{Generalized model for an arbitrary $m_0$}
\begin{itemize}
\item We start with no vertices and no edges.
\item Let $n\geq 0$ be the number of the current step. Put $N:=\left \lfloor \frac{n}{m_0} \right \rfloor$.
\item If $n\equiv 0\mod m_0$, then add one new vertex $v_N$.
\item Connect the vertex $v_N$ to exactly one of the vertices $v_0,\dots,v_{N-1}$, if this set is not empty.
\item A vertex $v_i,\ 0\leq i\leq N-1$ is chosen with probability
$$\frac{\w(\deg(v_i))}{\sum_{j=0}^{N-1} \w(\deg(v_j))}.$$
\end{itemize}
The basic, linear and convex models are just special cases of the generalized model, so it is enough to specify the generalized model for an arbitrary $m_0$.

\section{Pairwise vertex degree analysis}
\label{sect:2dim}
In this section we investigate a random walk on the two-dimensional integer lattice. In terms of preferential attachment, we consider two fixed vertices, and we are interested only in steps on which the degree of one of these vertices increases. We obtain a random walk by putting a point on $\n^2$, whose coordinates are equal to the degrees of these two vertices.

Consider the following random walk $R_k$ on $\N^2$. From the point $(A,B)$ it moves either to the point $(A+1,B)$ with probability $\tfrac{\w(A)}{\w(A)+\w(B)}$ or to the point $(A,B+1)$ with probability $\tfrac{\w(B)}{\w(A)+\w(B)}$. Note that the sum of the coordinates of $R_k$ increases by $1$ on every step.

\subsection{The number of paths}
\label{path_number}
We are interested in the probability that $R_k$ moves from some fixed point to the diagonal $\{(m,m)\}_{m \in \n}$. It means that the degrees of the two considered vertices become equal.

The event \{$R_k$ crosses the diagonal\} can be partitioned into events \{$R_k$ moves to the point $(m,m)$, and this is the first time it crosses the diagonal\}$_{m \in \N}$. We evaluate the probabilities of these events. To do it, we first need to count all admissible paths connecting the initial point and the point $(m,m)$, where by {\it admissible} we mean that only the endpoints of this path may belong to the diagonal.

\begin{lemma}
 Let $m\ge A>B$ be some natural numbers. By $\GG(A,B,A',B')$ denote the number of admissible paths connecting $(A,B)$ to $(A',B')$.

 Then
$$
\GG(A,B,m,m) = \frac{(2m-1-A-B)! (A-B)}{(m-A)!(m-B)!}\ .
$$

\end{lemma}
\begin{proof}
  By $\AA(A,B,A',B')$ denote the number of different up-right paths connecting the point $(A,B)$ with the point $(A',B')$. Then

\[
     \AA(A,B,A',B')={{A'+B'-A-B} \choose {A'-A}}\ ,
\]
because this is the number of ways to choose on which of $A'+B'-A-B$ steps the path goes up, and on the remaining steps the path goes to the right.

By
$$
\BB(A,B,A',B'):=\AA(A,B,A',B')-\GG(A,B,A',B')
$$
 denote the number of non-admissible paths between these two points.

To evaluate $\GG(A,B,m,m)$ we use André's reflection principle. Let us show that there is a one-to-one correspondence between all paths from $(A,B)$ to $(m-1,m)$ and all non-admissible paths from $(A,B)$ to $(m,m-1)$. Consider an arbitrary path between $(A,B)$ and $(m-1,m)$. It crosses the diagonal, because $A>B$ but $m-1 <m$. Now we perform the following operation: all steps before the intersection with the diagonal remain the same while all steps after the intersection are inverted (right $\leftrightarrow$ up). The part of the path after the intersection connected the point $(k,k)$ and the point $(m,m-1)$ for some $k$. Therefore, after the inversion it connects the point $(k,k)$ and the point $(m-1,m)$. Hence, now we have a non-admissible path from $(A,B)$ to $(m,m-1)$. This process can be reversed, because the first intersection point with the diagonal remains the same, hence the required bijection is constructed.

We get a formula
\[
   \BB(A,B,m,m-1)=\AA(A,B,m-1,m).
\]

Since all admissible paths from $(A,B)$ to $(m,m)$ must have an inner point $(m,m-1)$, we get the following chain of equalities, which concludes the proof:

\begin{eqnarray*}
  \GG(A,B,m,m) &=& \GG(A,B,m,m-1)
  \\
  &=& \AA(A,B,m,m-1)- \BB(A,B,m,m-1)
  \\
   &=& \AA(A,B,m,m-1)- \AA(A,B,m-1,m)
  \\
  &=& {{2m-1-A-B} \choose {m-A}} -  {{2m-1-A-B}\choose {m-A-1}}
  \\
  &=& \frac{(2m-1-A-B)!}{(m-A)!(m-B-1)!}
      - \frac{(2m-1-A-B)!}{(m-A-1)!(m-B)!}
  \\
  &=& \frac{(2m-1-A-B)!}{(m-A-1)!(m-B-1)!}
      \left(\frac{1}{m-A}- \frac{1}{m-B} \right)
  \\
  &=& \frac{(2m-1-A-B)! (A-B)}{(m-A)!(m-B)!}\ .
\end{eqnarray*}
\end{proof}

\subsection{The upper bound for the diagonal intersection probability}
By $q(A,m)$ denote the probability that $R_k$ moves from the point $(A,1)$ to the point $(m,m)$ following an admissible path.

\begin{prop}
  \label{thm:qAbound}
There exists a polynomial (with coefficients depending only on the weight function $\w$) $P(\cdot)$ such that for sufficiently large  $A$ and for any $m \ge A$ it is true that
\[
q(A,m) < \frac {P(A)}{2^A m^{3/2}}
\ .\]
\end{prop}
\begin{proof}
We evaluate upper bounds for the number of paths $\GG(A,1,m,m)$ and for the probability of every fixed path from $(A,1)$ to $(m,m)$ separately.

\begin{lemma}
  There exists a polynomial  $P_1(\cdot)$ such that
\[
\GG(A,1,m,m) \leq \frac{P_1(A)\ 2^{2m}}{2^Am^{3/2}}\ \ \ \forall\ A,m\ge A
\ .\]
\end{lemma}
\begin{proof}
\begin{eqnarray*}
&&
 \GG(A,1,m,m) = \frac{(2m-2-A)! (A-1)}{(m-A)!(m-1)!}\\
&=& \frac{(2m-2)!}{(m-1)!(m-1)!} \cdot \frac{A-1}{2m-1-A}
\cdot \frac{(m-A+1)\cdot \ldots \cdot (m-1)}
{(2m-A)\cdot \ldots \cdot (2m-2)}\, .
\end{eqnarray*}
In the last expression, the first fraction is a binomial coefficient. Note that the numerator and denominator of the last fraction have the same number of factors ($A-1$), and every factor of the numerator is at most the half of the corresponding factor of the denominator. Therefore
\begin{eqnarray*}
  \GG(A,1,m,m) \leq \frac{2^{2m}}{\sqrt{m}} \cdot \frac{P_1(A)}{m}\cdot \frac{1}{2^A}
\end{eqnarray*}
(all appearing constants are already included in the polynomial). The lemma is proved.
\end{proof}

\begin{lemma}
  There exist a polynomial $P_2(\cdot)$  and a number $A_1$ such that if $m \ge A>A_1$ then for every path $S$ from $(A,1)$ to $(m,m)$ it is true that
\[
p(S) \leq \frac{P_2(A)}{2^{2m}}
\ .\]

\end{lemma}
\begin{proof}
  Consider a composite path consisting of two simple paths:
$$S^*=S_1,S_2$$
where
$$ S_1=(A,1),(A,2),\ldots,(A,A),$$
$$S_2=(A,A),(A+1,A),(A+1,A+1),(A+2,A+1),(A+2,A+2) \ldots,(m,m).$$
\begin{prop}
  Among all paths with the same endpoints $S^*$ has the largest probability.
\end{prop}
\begin{proof}
  The probabilities of any two paths with the same endpoints are two fractions with same numerators but with different denominators. Therefore it is sufficient to find the path with a minimal denominator. Every denominator is a product of several expressions of the form  $\w(A_k)+\w(B_k)$ where $A_k+B_k$ is fixed. Due to the convexity of $\w$, the smaller $|A_k-B_k|$ is the smaller $\w(A_k)+\w(B_k)$ is. Clearly, the path $S^*$ minimizes $|A_k-B_k|$ on each step.
\end{proof}

By $\P(S)$ denote the probability of the path $S$.

Obviously, we have an upper bound for $\P(S_2)$:
$$
\P(S_2) \le \frac{1}{2^{2(m-A)}}=\frac{2^{2A}}{2^{2m}}.
$$

Now to conclude the lemma proof it suffices to show that
$$
\P(S_1) \le \frac{P_2(A)}{2^{2A}}.
$$
for some polynomial $P_2(A)$ and sufficiently large $A$.

The explicit formula for $\P(S_1)$ looks as follows:
\[
\P(S_1)=\frac{\w(1)}{\w(1)+\w(A)}\frac{\w(2)}{\w(2)+\w(A)}\ldots\frac{\w(A-1)}{\w(A-1)+\w(A)}
\ .\]

We introduce some notations. By $\ww_A(\cdot)$ denote a function that interpolates $\w$ at the points $1$ and $A$ in a linear way.

$\ww_A(n)$ has a form  $\ww_A(n)=k_An+b_A$ for some real numbers $k_A$ and $b_A$. Let $\beta(A):=b_A/k_A\in[-\infty,+\infty]$.

Choose $A_0\in\n$ such that $\w(A_0)>\w(1)$, and put $\beta_0:=\beta(A_0)$. Then for any $A>A_0$ it is obvious that
\begin{equation}
\label{eq:beta}
  -1<\beta(A)<\beta_0\ .
\end{equation}

\begin{remark}
\label{remark:beta}
The function $\beta(A)+1$ is not necessarily separated from zero, unlike the linear model. On the contrary, if $A=o(\w(A))$ then $\lim_{A\to\infty}\beta(A)=-1$.
\end{remark}

Every fraction here will increase if we replace $\w(k)$ by $\ww_A(k)$, because all fractions are less than $1$, and we add the non-negative number $\ww_A(k)-\w(k)$ to both numerator and denominator. Therefore,

\[
\P(S_1)\le
\frac{\ww_A(1)}{\ww_A(1)+\ww_A(A)}\frac{\ww_A(2)}{\ww_A(2)+\ww_A(A)}\ldots\frac{\ww_A(A-1)}{\ww_A(A-1)+\ww_A(A)}
\ .\]

We know that $\ww_A(n)=k_An+b_A$. After substituting it and reducing all fractions by $k_A$ we get
\[
\P(S_1)\le
\frac{1+\beta(A)}{1+A+2\beta(A)}\frac{2+\beta(A)}{2+A+2\beta(A)}
\ldots\frac{A-1+\beta(A)}{2A-1+2\beta(A)}
\ .
\]

Now we want to replace $\beta(A)$ by a larger number $\beta_0$. Note that if $\beta(A)<\beta_0$, then for any $B,C,D,E \in \n$ such that $D+E\beta(A)>0$ the following holds:
\begin{eqnarray}
\label{convex:monotone}
\frac{B+C\beta(A)}{D+E\beta(A)}<\frac{B+C\beta_0}{D+E\beta_0} \Leftrightarrow BE-CD<0 .
\end{eqnarray}

The condition on the right-hand side is satisfied, thus after replacing $\beta(A)$ by $\beta_0$ we get the following inequality:

\begin{eqnarray*}
\P(S_1)&\le&\frac{(1+\beta_0)(2+\beta_0)
\ldots (A+\beta_0-1)}{(A+1+2\beta_0) \ldots (2A-1+2\beta_0)}\\
&=&\frac{\Gamma(A+\beta_0)\Gamma(A+2\beta_0+1)}{\Gamma(\beta_0+1)\Gamma(2A+2\beta_0)}
\ .
\end{eqnarray*}

By Stirling's formula for any  $z\geq 1$ it is true that
 $\Gamma(z+1) \asymp \sqrt{z} (\frac{z}{e})^z $.
 After applying this and hiding all the constants into the polynomial we get
\begin{eqnarray*}
\P(S_1) &\leq& \frac{P_4(A) e^{2A+2\beta_0}}{e^{A+2\beta_0}e^{A+\beta_0}}
\frac{(A+\beta_0-1)^{A+\beta_0-1}(A+2\beta_0)^{A+2\beta_0}}{(2A+2\beta_0-1)^{2A+2\beta_0-1}} \\
&\leq& P_3(A) \cdot \left(\frac{A+\beta_0-1}{2A+2\beta_0-1}\right)^{A+\beta_0-1} \cdot \left(\frac{A+2\beta_0}{2A+2\beta_0-1}\right)^{A+2\beta_0}\\
&=&P_3(A)\cdot\frac{1}{2^{2A+3\beta_0-1}}\cdot \left(\frac{A+\beta_0-1}{A+\beta_0-1+1/2}\right)^{A+\beta_0-1} \cdot \left(\frac{A+2\beta_0}{A+2\beta_0-(1/2+\beta_0)}\right)^{A+2\beta_0}\\
&\leq& P_2(A) \cdot \frac{1}{2^{2A}}
\ .
\end{eqnarray*}
The last inequality is not as obvious as the other ones. Note that
\[
\left(\frac{x}{x+a}\right)^x=\left(1-\frac{a}{x+a}\right)^x\le \exp(-ax/(a+x))
\]
and that for large $x$ and bounded $a$ this expression is also bounded by some constant, which has also been already included into the polynomial.
\end{proof}

The conclusion of Proposition \ref{thm:qAbound} follows from our lemmas by multiplication of the corresponding inequalities.
\end{proof}

\begin{cor}
\label{qAbound}
By $q(A)$ denote the probability that our random walk moves from the point $(A,1)$ to the diagonal. Then for sufficiently large values of $A$ and for some polynomial $P(\cdot)$ it is true that
$$q(A) < \frac{P(A)}{2^A}\ .$$
\end{cor}
\begin{proof}
By Proposition \ref{thm:qAbound} we get that
\[
	  q(A)\le \sum_{m=A}^\infty q(A,m)
 	  \le \frac{P(A)}{2^A} \sum_{m=A}^\infty \frac1{m^{3/2}}\ .
 \]
  It remains to note that the series $\sum\frac1{m^{3/2}}$ is convergent.
\end{proof}

\subsection{Limit distribution of the random walk in the linear case}
Suppose $\w(n)=n+\beta$, $\beta>-1$. In this case, the following proposition provides an explicit  asymptotic form of the random walk distribution.:

\begin{prop}
\label{betadistr}
If $R_k$ starts at the point $(A,1)$ then the quantity $A_k/(A_k+B_k)$ tends to some random variable $H(A)$ as $k$ tends to infinity. Moreover, $H(A)$ has a beta probability distribution:

$$H(A) \sim Beta(1+\beta,A+\beta)\ .$$
\end{prop}
\begin{proof}
As noted in \cite{B}, this random walk is a special case of Pólya urn model with initial parameters $(1+\beta,A+\beta)$. Recall that for the urn model the limit distribution of that fraction is well known, see, for example, \cite{M} or \cite{JK}.
\end{proof}

\section{The proof of the main result}

\label{sect:main}
\subsection{Comparison between the convex model and the linear model}
\label{subsect:lemmas}
Motivated by (\ref{eq:beta}), for the convex model with weight function $\w$ we introduce the linear model with $\ww(n)=n+\beta_0$ and call it {\it the linear comparison model}.

\begin{lemma}[The comparison lemma]

\label{convex:compare}
Suppose in the convex model before the $n$-th step the vertex $v_t$ has the maximal degree $m$, and the degrees of all other vertices are fixed. Let $m>A_0$. By $p$ denote the probability that the next edge is attached to the vertex $v_t$. Now consider exactly the same situation (all the degrees remain the same), but in the linear comparison model. By $\pp$ denote the probability that the next edge is attached to the vertex $v_t$ in the linear case.

Then $p \ge \pp$.
\end{lemma}
\begin{proof}

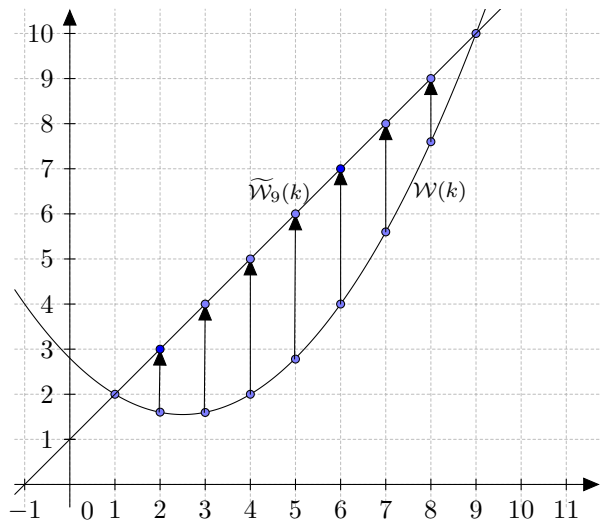
\begin{figure}
  \centering
\definecolor{qqqqff}{rgb}{0.0,0.0,1.0}
\definecolor{xdxdff}{rgb}{0.49019607843137253,0.49019607843137253,1.0}
\definecolor{cqcqcq}{rgb}{0.7529411764705882,0.7529411764705882,0.7529411764705882}
\begin{tikzpicture}[line cap=round,line join=round,>=triangle 45,x=0.6cm,y=0.6cm]
\draw [color=cqcqcq,dash pattern=on 1pt off 1pt, xstep=0.6cm,ystep=0.6cm] (-1.220000000000001,-0.5000000000000002) grid (11.740000000000004,10.540000000000004);
\draw[->,color=black] (-1.220000000000001,0) -- (11.740000000000004,0);
\foreach \x in {-1,1,2,3,4,5,6,7,8,9,10,11}
\draw[shift={(\x,0)},color=black] (0pt,2pt) -- (0pt,-2pt) node[below] {\footnotesize $\x$};
\draw[->,color=black] (0,-0.5000000000000002) -- (0,10.540000000000004);
\foreach \y in {,1,2,3,4,5,6,7,8,9,10}
\draw[shift={(0,\y)},color=black] (2pt,0pt) -- (-2pt,0pt) node[left] {\footnotesize $\y$};
\draw[color=black] (0pt,-10pt) node[right] {\footnotesize $0$};
\clip(-1.220000000000001,-0.5000000000000002) rectangle (11.740000000000004,10.540000000000004);
\draw[smooth,samples=100,domain=-1.220000000000001:11.740000000000004] plot(\x,{14/5+1/5*(\x)^(2)-(\x)});

\begin{scriptsize}
\draw[color=black] (8.2,6.5) node {$\w(k)$};
\draw[color=black] (4.65,6.5) node {$\ww_9(k)$};
\draw [fill=xdxdff] (1,2) circle (1.5pt);
\draw [fill=xdxdff] (9,10) circle (1.5pt);
\draw [fill=xdxdff] (8,7.600000000000001) circle (1.5pt);
\draw [fill=xdxdff] (7,5.600000000000001) circle (1.5pt);
\draw [fill=xdxdff] (6,4) circle (1.5pt);
\draw [fill=xdxdff] (5,2.7800800000000008) circle (1.5pt);
\draw [fill=xdxdff] (4,2) circle (1.5pt);
\draw [fill=xdxdff] (3,1.5960800000000002) circle (1.5pt);
\draw [fill=xdxdff] (2,1.6040799999999997) circle (1.5pt);

\end{scriptsize}

\draw [domain=-1.220000000000001:11.740000000000004] plot(\x,{(-8-8*\x)/-8});

\draw [->] (1.98,1.6040799999999997) -- (2,3);
\draw [->] (2.98,1.5960800000000002) -- (3,4);
\draw [->] (4,2) -- (4,5);
\draw [->] (4.98,2.7800800000000008) -- (5,6);
\draw [->] (6,4) -- (6,7);
\draw [->] (7,5.600000000000001) -- (7,8);
\draw [->] (8,7.600000000000001) -- (8,9);
\begin{scriptsize}
\draw [fill=qqqqff] (2,3) circle (1.5pt);
\draw [fill=xdxdff] (3,4) circle (1.5pt);
\draw [fill=xdxdff] (4,5) circle (1.5pt);
\draw [fill=xdxdff] (5,6) circle (1.5pt);
\draw [fill=qqqqff] (6,7) circle (1.5pt);
\draw [fill=xdxdff] (7,8) circle (1.5pt);
\draw [fill=xdxdff] (8,9) circle (1.5pt);
\end{scriptsize}
\end{tikzpicture}
\caption{Replacing the weight function increases the relative weight of all vertices except for the current hub. \label{fig:convex}}

\end{figure}

Note that for all $1\le k \le m$, $\w(k) \le \ww_m(k)$ due to the convexity of $\w$. Then
\begin{eqnarray*}
p&=&\frac{\w(m)}{\sum_v(\w(\deg(v)))} \ge \frac{\ww_m(m)}{\sum_v(\ww_m(\deg(v)))} \\
&=&\frac{m+\beta(m)}{2(n-1)+n\beta(m)} \ge \frac{m+\beta_0}{2(n-1)+n\beta_0}=\pp \ .
\end{eqnarray*}

First we increase the denominator (see Figure \ref{fig:convex}), then we reduce the fraction, and then we use (\ref{convex:monotone}).

\end{proof}

\begin{remark}
Unlike the expressions on the left hand side, $\pp$ depends only on $m$ and on the total degree of vertices.
\end{remark}

\subsection{The maximal degree grows fast enough}
\label{subsect:tooFast}
In the linear and basic models the degree of any fixed vertex grows fast enough to provide the convergence of the series $\sum q(A)$ with probability $1$, see Remark \ref{rem:nomax} below. Unfortunately, this is not always the case in the convex model, for example, if $\w(n)=2^{2^n}$ then with positive probability the degree of the first vertex will be bounded, because the second one will be connected to almost all vertices. So, any fixed vertex degree can be bounded. However, the maximal degree (the degree of random vertex), as we will see, grows fast enough with probability $1$.

By $\M_n$ denote the maximal degree before the $n$-th step.

\begin{prop} \label{thm::fastmax}
There exists a sequence $C_n$ of positive real numbers satisfying the following conditions:
\begin{enumerate}
  \item $C_n$ grows fast enough: the expression $C_n n^{-1/(4+2\beta_0)}$ converges to a positive finite limit,
  \item $C_n/\M_n$ is a supermartingale with respect to the filtration

  $\sigma_n=\sigma(\M_1, \dots, \M_n)$.
\end{enumerate}
\end{prop}
\begin{cor}
There exists a (random) constant $M$ such that for all $n\geq 2$,
\begin{equation}
 			\label{fastM}
 		 \M_n\ge M n^{1/(4+2\beta_0)}.
 \end{equation}

\end{cor}

\begin{proof}

$C_n/\M_n$ is a positive supermartingale, hence by Doob's theorem it tends to a finite limit with probability $1$, therefore this sequence with probability $1$ is bounded by some random variable $C$. But this implies $\M_n \ge C_n/C$ with probability $1$, i.e. with probability $1$ for all $n \ge 2$ we get (\ref{fastM}).
\end{proof}

\begin{proof}[Proof of Proposition \ref{thm::fastmax}]
By $V_n$ denote the set of vertices before the $n$-th step, and by $p_n$ denote the probability that the maximal degree increases on the $n$-th step. We can bound it from below:
\[
p_n \ge \frac{\w(\M_n)}{\sum_{v\in V_n} \w(\deg(v))}\ge \frac{\ww_{\M_n}(\M_n)}{\sum_{v\in V_n} \ww_{\M_n}(\deg(v))}  \geq \frac{\M_n+\beta_0}{\www_n}
=:\pp_n\ .\]
Here $\www_n=2(n-1)+n\beta_0$.

Denote $\alpha=4+2\beta_0$.

For the sequence $Y_n:=C_n/\M_n$ to be a supermartingale it is necessary to show that
\[
\E(Y_{n+1}|\FF_n) \le Y_n
\ .\]

Note that
\begin{eqnarray*}
Y_{n+1}/C_{n+1}
&=&
\begin{cases}
\frac1{\M_n+1}& \mbox{with probability } p_n,\\
\frac1{\M_n}& \mbox{with probability } 1-p_n .\\
\end{cases} \\
\end{eqnarray*}

It follows that
\begin{eqnarray*}
\E(Y_{n+1}/C_{n+1}|\FF_n)
&=&\frac{p_n}{\M_n+1}+\frac{1-p_n}{\M_n}\\
&=&\frac{p_n\M_n+\M_n+1-p_n\M_n-p_n}{\M_n(\M_n+1)}\\
&=&\frac{\M_n+1-p_n}{\M_n(\M_n+1)}
=\frac{1}{\M_n}-\frac{p_n}{\M_n(\M_n+1)}
\\
&\le&\frac{1}{\M_n}-\frac{\pp_n}{\M_n(\M_n+1)}
\le \frac{1}{\M_n}-\frac{\pp_n}{2\M_n^2}
\\
&\le&\frac{1}{\M_n}-\frac{1+\beta_0/\M_n}{2\M_n\www_n}
\le\frac{1}{\M_n}-\frac{1}{2\M_n\www_n}
\\
&=&\frac{1}{\M_n}\left(1-\frac{1}{2(2(n-1)+n\beta_0)}\right)
\\
&=&\frac{1}{\M_n}\left(1-\frac{1/(4+2\beta_0)}{n-4/(4+2\beta_0))}\right)
\\
&=&\frac{1}{\M_n}\left(1-\frac{\alpha}{n-4\alpha}\right)
=\frac{1}{\M_n}\left(\frac{n-5\alpha}{n-4\alpha}\right)\ .
\end{eqnarray*}

Now it is clear that the following inequality is sufficient for $Y_n$ to be a supermartingale:
\[
\frac{C_{n+1}}{\M_n}\left(\frac{n-5\alpha}{n-4\alpha}\right) \le \frac{C_n}{\M_n}
\ .\]
To make this inequality true put, for example, $C_{n+1}=C_n(1+\frac{\alpha}{n-5\alpha})$

A simple computation shows that the sequence $C_n/n^{\alpha}$ has a positive and finite limit, therefore $C_n$ satisfies both conditions from the statement of the proposition. This completes the proof.

%
%
%
%

\end{proof}

\subsection{Finite number of hubs}
\label{subsect:finiteCandidates}
In this subsection we prove that the set of vertices that have been hubs on some step is finite with probability $1$.

Consider a set of events
$$
B_M=\{\forall n\ \M_n>M n^{1/(4+2\beta_0)}\}\ .
$$
for any real $M>0$.

Let $v_{l(n)}$ be some vertex which has the maximal degree before the $n$-th step (in general, there can be several such vertices). Consider the following event
$$H_n=\{\textrm{the vertex }v_{n+1}\textrm{ has the same degree as }v_{l(n)}\textrm{ on some future step}\}.$$
We recall that the joint behaviour of vertices $v_{l(n)}$ and $v_{n+1}$ is described by the random walk from Section \ref{sect:2dim}, starting from the point $(\M_n,1)$. Using Corollary \ref{qAbound}, we get that for any $M>0$ and for sufficiently large $n$ the following is true:

	\[
 		\P(H_n \cap B_M) \le \max_{A\ge M n^{1/(4+2\beta_0)}} \frac{P(A)}{2^{A}}
 		\le \frac{P_1(M n^{1/(4+2\beta_0)})}{2^{M n^{1/(4+2\beta_0)}}}\ .		
 	\]
 where $P,\ P_1$ are some polynomials. The expressions on the right hand side form a convergent series, therefore, using Borel--Cantelli lemma one can show that the event $H_n \cap B_M$ occurs for only finitely many indices $n$ with probability $1$. Moreover, because of \eqref{fastM}, we see that $P(B_M)\to 1$ as $M \to 0$. Therefore, the event $H_n$ also occurs for only finitely many indices $n$ with probability $1$.

 Hence only finitely many vertices have been hubs.

 \begin{remark} \label{rem:nomax}
 In the linear and basic models the proof can be simplified using any fixed vertex for comparison instead of the leader $l(n)$, because in these models even the degree of any fixed vertex grows fast enough, i.e. like some power of $n$.
 \end{remark}

\subsection{Finite number of leader changes between any two fixed vertices}
\label{subsect:finitePair}
It remains to prove the following result:

\begin{thm}
  For any two vertices the set of all steps on which their degrees coincide is finite.
\end{thm}
\begin{proof}
  Consider any two vertices and the corresponding two-dimensional random walk. Suppose the random walk starts from the point $(A_k,B_k)$, which means that the degrees of these two vertices were at first equal to $A_k$ and $B_k$ respectively. Without loss of generality we may assume that $A_k+B_k > A_0$. Consider the linear two-dimensional comparison random walk (according to Lemma \ref{convex:compare}) starting from the same point, but with the other weight function $\ww(A)=A+\beta_0$.

  First we introduce some notation. By $\Delta_n:=|A_n-B_n|$ denote the difference between $A_n$ and $B_n$, and by $\dd_n:=|\aa_n-\bb_n|$ denote the corresponding difference in the linear comparison model. 

\begin{prop}
There exists a coupling between $\Delta_n$ and $\dd_n$, such that $\Delta_n$ stochastically dominates $\dd_n$ for any $n\ge k$.

\end{prop}

\begin{proof}
  We construct both functions $\Delta_n$ and $\dd_n$ on the same probability space preserving every independency relation each of them must satisfy.

Using induction on $n$, we now show that for every $n$, $\Delta_n \ge \dd_n$ with probability $1$. For $n=k$ it is true.

Now consider a set $L \subset \Omega$ of positive measure $p$ such that the functions $\Delta_n$ and $\dd_n$ are constants on $L$, and, by induction, $\Delta_n \ge \dd_n$ on $L$.

By $q$ denote the probability that $\Delta_n$ increases by $1$ (therefore, it decreases by $1$ with probability $1-q$), and by $\q$ denote the probability that $\dd_n$ increases by $1$. Let $\dd_n$ be positive on $L$. Then, by Lemma \ref{convex:compare}, $q>\q$. Let $L'$ be a subset of $L$ on which $\Delta_{n+1}=\Delta_n+1$, and $\L'$ be a subset of $L$ on which $\dd_{n+1}=\dd_n+1$. Clearly, the probability of the set $L'$ is greater than the probability of the set $\L'$, therefore we can choose them in such a way that $\L' \subset L'$. So on $L$ the induction inequality $\Delta_{n+1} \ge \dd_{n+1}$ now holds.

The only remaining set is the set where $\dd_n=0$. On its subset where $\Delta_n\neq 1$ the required inequality $\Delta_{n+1} \ge \dd_{n+1}$ holds automatically, and now all we need is to note that $\Delta_n$ and $\dd_n$ are of the same parity (because their parities both change on every step), so the set where $\dd_n=0$ and $\Delta_n=1$ is empty. This concludes the construction of the functions $\Delta_n$ and $\dd_n$.

\end{proof}
Now we show that with probability $1$ the sequence $\dd_n$ is equal to zero only finitely many times. Then it is also true for $\Delta_n$, because $\Delta_n \ge \dd_n$.

It follows from Proposition \ref{betadistr} and from the absolute continuity of beta-distribution that the probability of every particular value equals to zero. Therefore with probability $1$ $A_n/(A_n+B_n)$ converges to some $y \neq \frac12$. Hence this fraction can be equal to $\frac12$ only finitely many times, and it means that $\dd_n$ equals to zero only finitely many times with probability $1$, q.e.d.

\end{proof}

Now the result of Theorem \ref{thm:main} for $m_0=1$ obviously follows from those of Subsections \ref{subsect:finiteCandidates} and \ref{subsect:finitePair}.

From Theorem \ref{thm:main} we can easily deduce an important known result about the behaviour of maximal degrees in the linear model from \cite{M2}:

\begin{cor}
  In the linear model the variable $\M_n$ satisfies the following:
\[
\M_n n^{-1/(2+\beta)} \to \mu
\ ,\]
where $\mu$ is an almost surely positive and finite random variable.
\end{cor}
\begin{proof}
  We know that $\M_n$ behaves like the degree of some fixed vertex. Moreover, it is known that in the linear model the degree of every vertex is asymptotically equivalent to $n^{-1/(2+\beta)}$ multiplied by some random constant.
\end{proof}

\subsection{Generalization to the case of an arbitrary $m_0$}
\label{subsect:m0}
The case of $m_0>1$ is often considered to be much more complicated than the case of $m_0=1$, because the graph is not a tree. But it turns out that all steps of the presented proof (pairwise vertex degree analysis, the sufficiently fast growth of the maximal degree, finite number of hubs and leader changes) remain literally the same for $m_0>1$, except for just one change: in the pairwise analysis part, the random walk related to the degrees of the vertices starts not from the point $(A,1)$, but from the point $(A,m_0)$. This obstacle can be easily avoided by introducing a new (convex and unbounded) weight  function $\w'(n):=\w(n+m_0-1)$. Then the random walk with the weight function $\w$ starting from the point $(A,m_0)$ is isomorphic to the random walk with the weight function $\w'$ starting from the point $(A-m_0+1,1)$, and for this case we have already provided all necessary bounds.

\section*{Acknowledgements}

I would like to thank my supervisor Professor M. A. Lifshits for his guidance, patience, time and valuable remarks. I am also grateful to Andrey Alpeev, who suggested to consider the convex model, and to Vladimir Zolotov for many fruitful conversations.
This research is supported by JSC "Gazprom Neft" and by Chebyshev Laboratory  (Department of Mathematics and Mechanics, St. Petersburg State University)  under RF Government grant 11.G34.31.0026.

\bibliographystyle{my-plain}
\bibliography{rndgrph}

Saint-Petersburg State University

email: {\tt  pgalashin@gmail.com}

 \end{document}